\documentclass{amsart}
\usepackage{graphicx}
\usepackage{amscd}
\usepackage{amsmath}
\usepackage{amsfonts}
\usepackage{amssymb}
\theoremstyle{plain}
\newtheorem{theorem}{Theorem}[section]
\newtheorem{corollary}[theorem]{Corollary}

\newtheorem{proposition}[theorem]{Proposition}
\theoremstyle{definition}
\newtheorem{example}[theorem]{Example}

\newtheorem{definition}[theorem]{Definition}

\newtheorem{remark}[theorem]{Remark}
\theoremstyle{remark}

\newtheorem{notation}[theorem]{{\bf Notation}}

\begin{document}
\title{PSI-morphisms}

\author{Tiberiu Dumitrescu  and Mihai Epure }

\address{Facultatea de Matematica si Informatica,University of Bucharest,14 A\-ca\-de\-mi\-ei Str., Bucharest, RO 010014,Romania}
\email{tiberiu\_dumitrescu2003@yahoo.com, tiberiu@fmi.unibuc.ro, }

\address{Simion Stoilow Institute of Mathematics of the Romanian AcademyResearch unit 6, P. O. Box 1-764, RO-014700 Bucharest, Romania}\email{epuremihai@yahoo.com, mihai.epure@imar.ro}

\thanks{2020 Mathematics Subject Classification: Primary 13B10, Secondary 13B30.}

\begin{abstract}
\noindent 
We extend to ring morphisms the recent work of Mohamed Khalifa on PSI-extensions.
\end{abstract}

\maketitle

\section{Introduction}

In \cite{K}, Mohamed Khalifa  introduced and studied the so-called {\em prime submodule ideal ring extensions} (for short, { PSI-extensions}). 
An extension $A\subseteq B$ of integral domains is a {\em PSI extension} if every $A$-prime ideal of $B$ is  prime in $B$. Here  a proper ideal $I$ of $B$ is {\em $A$-prime} if  $ab\in I$ with $a\in A$ and  $b\in B$ implies that  $a\in I$ or $b\in I$.

The next theorem summarizes  a good deal of the results in \cite{K}.
Specifically, for $(i)$-$(iii)$, $(iv)$-$(v)$, $(vi)$, $(vii)$, $(viii)$, $(ix)$, $(x)$, $(xi)$, $(xii)$ see results (2.2), (2.3), (2.4), (2.7), (2.9), (2.11), (2.14), (3.2), (3.5) in \cite{K} respectively.

\begin{theorem} {\em (Khalifa)}\label{3}
Let $A  \subseteq B \subseteq C$  be  extensions of domains and $X$ an indeterminate.

$(i)$ If  $A \subseteq C$ is a PSI-extension, then so is $B \subseteq C$.

$(ii)$ For every multiplicative set $S$ of $A$, $A  \subseteq A_S$ is a PSI-extension.

$(iii)$ If $A \subseteq B$ is a PSI-extension, then so is 
 $A/(Q\cap A) \subseteq B/Q$ for each prime ideal $Q$ of $B$.

$(iv)$ If $A \subseteq B$ is a PSI-extension, then so is 
 $A_S \subseteq B_S$ for each   multiplicative set $S$ of $A$.

$(v)$ $A \subseteq B$ is a PSI-extension iff so is
 $A_P \subseteq B_P$ for each maximal ideal $P$ of $A$.

$(vi)$ $A \subseteq B$ is a PSI-extension  iff it is an  INC-extension and $\Omega(P)$ is prime in $B$ or $\Omega(P)=B$ for every prime ideal $P$ of $A$. Here $\Omega(P)=PB_P\cap B.$

$(vii)$ If $A \subseteq B$ have the same prime ideals, then $A \subseteq B$ is a PSI-extension.

$(viii)$ If $A \subseteq B$ is an integral PSI-extension and $B \subseteq C$ is a PSI-extension, then   $A \subseteq C$ is a PSI-extension.

$(ix)$ Suppose that $A$ is integrally closed in $B$ and let $\Gamma$ be the set of all intermediate rings between  $A$ and $B$.
Then $A \subseteq B$ is a PSI-pair , i.e. $A \subseteq T$ is a PSI-extension
for each $T\in \Gamma$ iff all residue field extensions of $A \subseteq T$ are algebraic for each $T\in \Gamma$.

$(x)$ If $D$ is a Pr\" ufer domain with quotient field $K$, then
$D + XK[[X]]\subseteq  K[[X]]$ is a PSI-pair.

$(xi)$  $A[X] \subseteq B[X]$ is a PSI-extension iff 
$A \subseteq B$ is a PSI-extension  and 
all residue field extensions of $A \subseteq B$ are trivial.

$(xii)$ Suppose that $A$ is a quasi-local integrally closed domain. Then 
$A[[X]] \subseteq T[[X]]$ is a PSI-extension for each overring $T$ of $A$ iff $A$ is a valuation domain.
\end{theorem}

The purpose of this paper is to extend a part of the facts in Theorem \ref{3} to the more flexible setup of commutative ring morphisms. By {\em ring} we always mean a commutative unitary ring and our ring morphisms are unitary. Our notation and terminology is standard like in \cite{Kp}.

We give our key definition which extends the concept of PSI-extension  of \cite{K} to ring morphisms.

\begin{definition}\label{5}
Let $u:A\rightarrow B$ be a ring morphism.

$(i)$  A proper ideal $I$ of $B$ is said to be {\em $A$-prime} if  $u(a)b\in I$ with $a\in A$ and  $b\in B$ implies that  $u(a)\in I$ or $b\in I$.
It follows easily that $u^{-1}(I)$ is a prime ideal of $A$.

$(ii)$ Say that $u$ is a {\em PSI-morphism} if every $A$-prime ideal of $B$ is  prime in $B$. We say that a ring extension $C\subseteq  D$ is a 
{\em PSI-extension} if  the inclusion map $C\hookrightarrow D$ is a PSI-morphism.
\end{definition}

In Section 2, we study the PSI-morphisms. We prove that the  PSI-morphisms are those morphisms whose fiber rings are either zero or  fields (Theorem \ref{2}). An immediate consequence is that the spectral map of a PSI-morphism is injective (Corollary \ref{7}). Some examples of PSI-extensions are given in Examples \ref{21} and \ref{210}. 
Our Proposition \ref{9} extends parts $(ii)$ and $(vii)$ of  Theorem \ref{3} and also shows that an epimorphism is a PSI-morphism. 
It is easy to identify the minimal ring extensions which are PSI
(Proposition \ref{23}). The class of PSI-morphisms is closed under map composition (Theorem \ref{8}), thus  extending parts $(i)$ and $(viii)$ of  Theorem \ref{3}.  Corollary \ref{10} extends parts $(iii$-$v)$ of Theorem \ref{3}.

In Section 3, we study a special type of PSI-morphism. Call a ring morphism a {\em strong PSI-morphism} if its fiber maps are isomorphisms.
An epimorphism is a strong PSI-morphism (Proposition \ref{14}) but not conversely (Example \ref{18}). 
A PSI-morphism is strong iff its residue field extensions are trivial
(Theorem \ref{13}). Examples of strong PSI-morphisms using  Nagata idealization rings are provided by Proposition \ref{22}.
 In Theorem \ref{16}, we show that the class of strong
PSI-morphisms is stable under base extension and derive that a ring morphism is strong PSI iff the corresponding polynomial morphism 
is a PSI-morphism, thus extending part $(xi)$ of Theorem \ref{3}.
 Consequently, we prove that a finite strong PSI-morphism is surjective 
(Corollary  \ref{300})  
and a finite type strong PSI-morphism is an epimorphism 
(Theorem  \ref{301}). The class of (strong) PSI-morphisms is  closed under map composition (Theorem \ref{12}) and inductive limits (Proposition \ref{24}). If $ A\subseteq B$ is a ring extension, we can talk about the greatest intermediate ring $A\subseteq C \subseteq B$ which is strong PSI over $A$ (Proposition  \ref{26}). 
An extension of rings $A\subseteq B$ sharing an ideal $I$  is a (strong) PSI-extension iff so is $A/I\subseteq B/I$ (Proposition \ref{17}).
Consequently, $A[[X]]\subseteq  B[[X]]$ is a PSI-extension when $A\subseteq B$ is ring extension such that some maximal ideal $M$ of $B$ is contained in $A$ (Proposition \ref{29}), thus extending part $(xii)$ of Theorem \ref{3}.

\section{PSI-morphisms}

We recall some standard notation.

\begin{notation}\label{15}
Let $u:A\rightarrow B$ be a ring morphism.  Recall that if $P$ is a prime ideal  of $A$, then {\em fiber ring} of $u$ at $P$ is the ring 
$k_A(P)\otimes_A B$, where $k_A(P)=A_P/PA_P$ is the {\em residue field} of $A$ at $P$.
We have the canonical injective map $Spec(k_A(P)\otimes_A B)\rightarrow Spec(B)$  whose image is the set of prime ideals of $B$ lying over $P$ in $A$.
So $k_A(P)\otimes_A B$ is nonzero iff there is some prime ideal  of $B$ lies over $P$. 
When $v:C\rightarrow D$ is a canonical ring morphism, we shall write $C=D$ to mean that $v$ is an isomorphism. For instance, if
$Q\in Spec(B)$ and $P=u^{-1}(Q)$, then
$$k_A(P)=k_A(P)\otimes_A B=k_B(Q)$$  means  that  the canonical maps
$$k_A(P)\rightarrow  k_A(P)\otimes_A B\rightarrow k_B(Q)$$
are isomorphisms.
\end{notation}

We start with two simple results about $A$-primes.

\begin{proposition}\label{4}
Let $u:A\rightarrow B$ be a ring morphism, $I$ a proper ideal of $B$ and $P=u^{-1}(I)$. The following  are equivalent.

$(i)$ $I$ is $A$-prime.

$(ii)$ $B/I$ is a torsion-free $A/P$-module.

$(iii)$ $I$ is a contracted ideal via the ring morphism 
$B\rightarrow k_A(P)\otimes_A B$.
\end{proposition}
\begin{proof}
As observed in Definition \ref{5}, $P$ is a prime ideal. The equivalence 
$(i)\Leftrightarrow (ii)$ follows easily from definitions.
Since $ k_A(P)$ is the quotient field of $A/P$, 
$(ii)$ holds iff the canonical map $B/I\rightarrow k_A(P)\otimes_A B/I$
is injective iff $(iii)$ holds.
\end{proof}

\begin{corollary}\label{6}
Let $u:A\rightarrow B$ be a ring morphism and $P$ a prime ideal of $A$. The following are equivalent.

$(i)$ All $A$-prime ideals of $B$ lying over $P$ are prime in $B$. 

$(ii)$  The fiber ring $k_A(P)\otimes_A B$ is a field.
\end{corollary}
\begin{proof}
Note that every ideal of $k_A(P)\otimes_A B$ is extended from $B$. So, by Proposition \ref{4}, $(i)$ holds iff every ideal of  $k_A(P)\otimes_A B$ is prime iff $(ii)$ holds.
\end{proof}

We give our key tool in dealing with PSI-morphisms. 

\begin{theorem}\label{2}
For a  ring morphism $u:A\rightarrow B$, the following are equivalent:

$(i)$ $u$ is a  PSI-morphism.

$(ii)$ For every $ P\in Spec(A)$ the fiber  $k_A(P)\otimes_A B$ is a field or the zero ring.

$(iii)$  $u(A)\subseteq B$ is a PSI-extension.

\end{theorem}
\begin{proof}
$(i)\Leftrightarrow (ii)$ is covered by  Corollary \ref{6}.
$(i)\Leftrightarrow (iii)$ follows observing that $u$ and the inclusion morphism $u(A)\hookrightarrow B$ have the same nonzero fiber rings.
\end{proof}

\begin{corollary}\label{7}
Let $u:A\rightarrow B$ be a PSI-morphism. Then

$(i)$ The spectral map $\alpha:Spec(B)\rightarrow Spec(A)$ is injective.

$(ii)$ If $P\in Im(\alpha)$, then 
$$\alpha^{-1}(P) =\{b\in B|\ u(s)b\in PB \mbox{ for some } s\in A-P\}:=Q$$ and  $k_A(P)\otimes_A B = k_B(Q)$. 

$(iii)$ If $M\in Max(A)$, then $MB\in Max(B)$ or $MB=B$. 

$(iv)$ If $A$ is a field, then so is $B$.
\end{corollary}
\begin{proof}
$(i)$ and $(ii)$.
Let $P\in Im(\alpha)$. Then $k(P)\otimes_A B$ is a field 
(cf. Theorem \ref{2}), so $Q=\alpha^{-1}(P)$ is the kernel   of the   morphism $B\rightarrow k(P)\otimes_A B$. A short computation shows that $Q$ has the  indicated value.
The canonical map $k_A(P)\otimes_A B\rightarrow k_B(Q)$ is an isomorphism since  the second ring is obtained from the first one by factorization and localization. $(iii)$ If $MB\neq B$, then  $k(M)\otimes_A B=B/MB$ is a field (cf. Theorem \ref{2}), so $MB$ is a maximal ideal of $B$. $(iv)$ follows from $(iii)$.
\end{proof}

\begin{example}\label{21}
Let $A$ be a one-dimensional domain and $B$ an overring of $A$. Then    
$A\subseteq B$ is a PSI-extension iff $MB\in Max(B)$ or $MB=B$ for all $M\in Max(A)$, cf. part $(iii)$ of Corollary  \ref{7}.

Consider the particular case 
$A=\mathbb{Z}[\sqrt{d}]$, $B=\mathbb{Z}[(1+\sqrt{d})/2]$ where 
$d\in \mathbb{Z}$, $\sqrt{d}\notin \mathbb{Q}$ and $d$ is one modulo $4$. Then $A\subseteq B$ is a PSI-extension iff $d$ is five modulo $8$. Indeed, if  $d$ is one modulo $8$, then the fiber ring $k_A(2,1+\sqrt{d})\otimes_A B$ is isomorphic to $\mathbb{Z}_2 \times \mathbb{Z}_2$. On the other hand, if  $d$ is one modulo $8$, then every maximal ideal of $A$ extends to a maximal ideal of $B$. See also Proposition \ref{23}.
\end{example}

\begin{example}\label{210}
If $K$ is a field,  then $K[X]\subseteq K[[X]]$ is a PSI-extension, because $K[[X]][K(X)]=K((X))$ and $K\otimes_{K[X]} K[[X]]=K$.
Meanwhile, $A=\mathbb{Z}[X]\subseteq \mathbb{Z}[[X]]=B$ is not a PSI-extension
because 
$$k_A(X-6)\otimes_A B \simeq 
\mathbb{Q}\otimes_\mathbb{Z} \mathbb{Z}[[X]]/(X-6)\simeq 
\mathbb{Q}_2\otimes \mathbb{Q}_3 $$ where $\mathbb{Q}_p $ is the $p$-adic number field.
\end{example}

Let $u:A\rightarrow B$ be a ring morphism, $v:B\otimes_A B\rightarrow B$ the canonical ring morphism  induced by $u$ and $I=ker(v)$.
Recall that module of differentials  $\Omega_{B/A}$ is the $B$-module 
$I/I^2$.  Recall also that  $u$ is called an   {\em epimorphism} if the canonical map $v$ is an isomorphism, that is, $I=0$. Thus   $\Omega_{B/A}=0$ if $u$ is an epimorphism.

\begin{proposition}\label{9}
A ring morphism $u:A\rightarrow B$ is a PSI-morphism in each  case below.

$(i)$ Every prime ideal of $B$ is extended from $A$.

$(ii)$ $u$ is an epimorphism.

$(iii)$ $u$ is surjective.

$(iv)$ $u$ is the canonical map from $A$ to a fraction ring of $A$.

\end{proposition}
\begin{proof}
We use Theorem \ref{2}. Let $P$ be a prime ideal of $A$. 
Suppose that $(i)$ holds. Since the hypothesis of $(i)$ is stable under factorization and localization, it is satisfied by  
$k_A(P)\rightarrow k_A(P) \otimes_AB$, so $k_A(P) \otimes_AB$ is a field.
If $(ii)$ holds, then  $k_A(P) = k_A(P)\otimes_A B$, cf.  \cite[Lemme 1.0' and Corollaire 1.3]{L}.  
The other assertions follow from $(i)$. 
\end{proof}

Note that a proper field extension $K\hookrightarrow L$ is a PSI-extension but not an epimorphism. 

Next we extend  Example \ref{21}. Note $\mathbb{Z}[\sqrt{d}]$ is an index two subring of $ \mathbb{Z}[(1+\sqrt{d})/2]$.
Recall that a proper ring extension $A\subset B$ is  
{\em minimal} \cite{FO} if there is no proper intermediate ring $A\subset C\subset B$. In this case, there exists a maximal ideal called the {\em crucial ideal} of $A\subset B$ such that $A_P=B_P$ for each $P\in Spec(A)-\{M\}$, cf. \cite[Th\'eor\`eme 2.2]{FO}. Clearly $A_M\subset B_M$.

\begin{proposition}\label{23}
Let $A\subset B$ be a minimal ring  extension with crucial ideal $M$. Then $A\subset B$ is a PSI-extension iff $A\subset B$ is not finite or $A\subset B$ is  finite and $M$ is a maximal ideal of $B$.
\end{proposition}
\begin{proof}
The assertion is a consequence of the following remarks. If $A\subset B$ is not finite, then $A \hookrightarrow B$ is a flat epimorphism cf. \cite[Theorem 2.2]{FO}; so $A \hookrightarrow B$ is a PSI-extension cf. Proposition \ref{9}. Suppose that $A\subset B$ is  finite. Then $MB$ is a maximal ideal of $B$ iff so is $M$ cf.
\cite[Theorem 3.3]{PP}.
\end{proof}

Our next result extends parts $(i)$ and  $(viii)$ of Theorem \ref{3}.

\begin{theorem}\label{8}
Suppose that $u:A\rightarrow B$ and  $v:B\rightarrow C$ are ring morphisms.

$(i)$ If $u$ and $v$ are PSI-morphisms, then so is $vu$.

$(ii)$ If $vu$ is a PSI-morphism, then so is $v$.
\end{theorem}
\begin{proof}
$(i)$.
Let $M$ be a prime ideal of $C$, $Q=v^{-1}(M)$, $P=u^{-1}(Q)$.
As $u$ and $v$ are PSI-morphisms, we have $k_A(P)\otimes_A B = k_B(Q)$ and 
$k_B(Q)\otimes_B C = k_C(M)$, cf. Corollary \ref{7}, so
$$k_A(P)\otimes_A C = 
k_A(P)\otimes_A B \otimes_B C = k_B(Q)\otimes_B C =  k_C(M).
$$
Apply Theorem \ref{2}.

$(ii)$ Let $M$ be a prime ideal of $C$, $Q=v^{-1}(M)$, $P=u^{-1}(Q)$.
As $vu$ is a PSI-morphism, we have $k_A(P)\otimes_A C = k_C(M)$, cf. Corollary \ref{7}. Since $k_B(Q)\otimes_B C=C_Q/QC_Q$ is obtained from 
$k_A(P)\otimes_A C=C_P/PC_P$ by factorization and localization, we get that 
$k_B(Q)\otimes_B C$ is a field. Apply Theorem \ref{2}.
\end{proof}

In $(ii)$  above, it may happen that $vu$ is a PSI-morphism, while $u$ is not. For instance, 
$\mathbb{Z}\subset \mathbb{Q}(i)$ is a PSI-extension, but  
$\mathbb{Z} \subset \mathbb{Z}[i]$ is not. A similar example is
$\mathbb{Z}[\sqrt{-7}] \subset \mathbb{Z}[(1+\sqrt{-7})/2] \subset \mathbb{Q}(\sqrt{-7})$, cf. Example \ref{21}.

\begin{corollary}\label{10}
Let $u:A\rightarrow B$ be a  ring morphism.

$(i)$ If $u$ is a PSI-morphism, then so is $A/I \rightarrow B/J$ for all   ideals  $I,J$  of $A$, $B$ respectively such that  $u(I)\subseteq J$. 

$(ii)$ If $u$ is a PSI-morphism, then so is $A_S \rightarrow B_T$
for all multiplicative sets
$S,T$    of $A$, $B$ respectively such that  
$u(S)\subseteq T$.

$(iii)$ $u$ is a PSI-morphism iff so is $A/P \rightarrow B/PB$ for each  
$P\in Spec(A)$.

$(iv)$ $u$ is a PSI-morphism iff so is $A_M \rightarrow B_M$ for each  
$M\in Max(A)$.

\end{corollary}
\begin{proof}
$(i)$ By Proposition \ref{9} and Theorem \ref{8}, we get that 
$A \rightarrow B/J$ is a PSI-morphism, hence so is  $A/I \rightarrow B/J$.
A similar argument works for $(ii)$.
Assertion $(iii)$ follows from Theorem \ref{2} and the fact that the fiber rings of $u$ is the union of all fiber rings of $A/P \rightarrow B/PB$ for  
$P\in Spec(A)$. A similar argument works for $(iv)$.
\end{proof}

\section{Strong PSI-morphisms}

In this section we study a special type of PSI-morphim.  The motivation for doing that comes from the fact that the class of PSI-morphisms is not closed under base extension. A simple example is the field  PSI-extension $\mathbb{R}\subset \mathbb{C}$ whose polynomial extension $\mathbb{R}[X]\subset \mathbb{C}[X]$ is not PSI, cf. part $(xi)$ of Theorem \ref{3}.

\begin{definition}\label{11}
A  ring morphism $u:A\rightarrow B$ is said to be a {\em strong PSI-morphism} if  
$k_A(P)= k_A(P)\otimes_A B$ for all   $P\in Spec(A)$ with 
$k_A(P)\otimes_A B\neq 0$.
\end{definition}

\begin{proposition}\label{14}
  An epimorphism is a strong PSI-morphism
\end{proposition}
\begin{proof}
Use the proof of part $(ii)$ of Proposition \ref{9}.
\end{proof}

The converse of Proposition \ref{14} is not true as shown by Example \ref{18}. 

\begin{example}\label{180}
It is easy to check directly that the diagonal map 
$\mathbb{Z}\hookrightarrow \mathbb{Z}[1/2]\times \mathbb{Z}_2$ is a strong  PSI-morphism. More generally, if $A$ is a domain and $b$ a nonzero nonunit of $A$, then the diagonal map $A\hookrightarrow A[1/b]\times A/bA$ is an epimorphism \cite[page 3-09]{R}, so it is a strong  PSI-morphism, cf. Proposition  \ref{14}.
\end{example}

The strong PSI-morphisms are the PSI-morphisms with trivial residue field extensions.

\begin{theorem}\label{13}
For a  ring morphism $u:A\rightarrow B$ the following are equivalent.

$(i)$ $u$  is a strong PSI-morphism.

$(ii)$ $u$ is a  PSI-morphism and $k_A(u^{-1}(Q))=k_B(Q)$ for each  $Q\in Spec(B)$.

$(iii)$ $k_A(P)\otimes_A B$ is a $k_A(P)$-vector space of dimension $\leq 1$ for each $P\in Spec(A)$.
\end{theorem}
\begin{proof}
Use Definition \ref{11}, Theorem \ref{2} and the fact that $k_A(P)\rightarrow k_A(P)\otimes_A B$ is an injective ring morphism.
\end{proof}

\begin{proposition}\label{22}
Let $A$ be a ring, $M$ an $A$-module and $A(+)M$ the Nagata idealization ring of $M$. The canonical map 
$A\rightarrow A(+)M$ is a strong PSI-morphism iff 
$k_A(P)\otimes_A M=0$ for all   $P\in Spec(A)$.
\end{proposition}
\begin{proof}
We have $k_A(P)\otimes_A A(+)M\simeq k_A(P)\oplus (k_A(P)\otimes_A M)$ as   
$k_A(P)$-vector spaces, so Theorem \ref{13} applies.
\end{proof}

We exhibit   a strong PSI-morphism which is not an epimorphism.

\begin{example}\label{18}
Let $u$ be the canonical map $\mathbb{Z}\rightarrow \mathbb{Z} (+)\mathbb{Q}/\mathbb{Z}:=B$. 
As $K\otimes_\mathbb{Z}  \mathbb{Q}/\mathbb{Z}=0$ for $K=\mathbb{Q}$ or $\mathbb{Z}_p$ with $p$ prime, it follows that $u$ is a strong PSI-morphism, cf.  Proposition \ref{22}.
On the other hand we have
$$ B\otimes_\mathbb{Z} B \simeq 
\mathbb{Z} \oplus \mathbb{Q}/\mathbb{Z} \oplus \mathbb{Q}/\mathbb{Z} \not\simeq \mathbb{Z} \oplus \mathbb{Q}/\mathbb{Z} = B
$$
as Abelian groups, so $u$ is not an epimorphism. 
Writing $ \mathbb{Q}/\mathbb{Z}$ as an inductive limit of the ciclic groups $\mathbb{Z}_{n!}$, it can be shown that  $\Omega_{B/\mathbb{Z}}$ is null, so
 the kernel of the canonical map $B\otimes_\mathbb{Z} B \rightarrow B$
is a nonzero idempotent nilideal, cf. \cite[Proposition 1.5]{L}. Note also that $B$ is not $\mathbb{Z}$-flat because it has nonzero torsion. We were not able to find an example of a strong PSI-morphism with nonzero module of differentials.
\end{example}

In the light of Theorem \ref{13},  part $(xi)$ of Theorem \ref{3} shows (in our terminology) that   a domain extension 
$A \subseteq B$ is  strong PSI  iff the polynomial extension 
$A[X] \subseteq B[X]$ is a PSI-extension. 
We extend this result by showing that the class of strong PSI-morphisms is stable under base extension.

\begin{theorem}\label{16}
For a ring morphism $u:A\rightarrow B$ the following are equivalent.

$(i)$ $u$ is a strong PSI-morphism.

$(ii)$ $u\otimes_A C:C\rightarrow C \otimes_A B$ is a strong PSI-morphism
for each ring morphism $v:A\rightarrow C$.

$(iii)$ $u\otimes_A C:C\rightarrow B \otimes_A C$ is a  PSI-morphism for each ring morphism $A\rightarrow C$.

$(iv)$ $u\otimes_A A[X]:A[X]\rightarrow B[X]$ is a  PSI-morphism.
\end{theorem} 
\begin{proof}
$(i)\Rightarrow (ii)$ Set $D=C \otimes_A B$. Let $M\in Spec(C)$ such that $k_C(M)\otimes_C D$ is nonzero and let $P=v^{-1}(M)$. Note that $k_A(P)\otimes_A B$ is nonzero, so $k_A(P)= k_A(P)\otimes_A B$. We have 
$$k_C(M)\otimes_C D =
k_C(M)\otimes_C C \otimes_A B =
k_C(M) \otimes_A B =
$$
$$
= k_C(M) \otimes_{k_A(P)} k_A(P) \otimes_A B =
k_C(M) \otimes_{k_A(P)} k_A(P) = k_C(M).
$$
%
The implications $(ii)\Rightarrow (iii) \Rightarrow (iv)$ are clear.

$(iv)\Rightarrow (i)$ Let $P\in Spec(A)$ such that $k_A(P)\otimes_A B$
is nonzero. As $A[X]\rightarrow B[X]$ is a  PSI-morphism, we get that 
$A\rightarrow B$ and $k_A(P)[X]\rightarrow k_B(Q)[X]$ are PSI-morphisms,
cf. parts $(i$-$ii)$ of Corollary \ref{10}.
Thus $ k_A(P) \otimes_A B = k_B(Q) = k_A(P)$ by Theorem \ref{2} and \cite[Lemma 3.1]{K}.
\end{proof}

If we add some finitness condition, a strong PSI-morphism gets a particular form.

\begin{corollary}\label{300}
A finite strong PSI-morphism is    surjective.
\end{corollary}
\begin{proof}
It suffices to consider the case of an injective   finite  strong PSI-morphism 
$A\hookrightarrow B$. Localizing (Theorem \ref{16}), we may assume that $A$ is local with maximal ideal $M$. As  $A\hookrightarrow B$ is a strong PSI-morphism, we get $A/M=B/MB$, hence $B=A+MB$, thus $A=B$ by Nakayama's Lemma.
\end{proof}

\begin{theorem}\label{301}
A finite type strong PSI-morphism is an epimorphism.
\end{theorem}
\begin{proof}
Let $u:A\rightarrow B$ be a finite type strong PSI-morphism,
$v:B\otimes_A B\rightarrow B$ the canonical ring morphism  induced by $u$ and $I=ker(v)$. 
Then $I$ is a finitely generated ideal of $B\otimes_A B$, so 
$\Omega_{B/A}=I/I^2$ is a finitely generated $B$-module, cf. 
\cite[Lemme 1.7]{A}.
Let $Q\in Max(B)$ and $P=u^{-1}(Q)$. Since $u$ is a strong PSI-morphism, we get $k_A(P)=k_A(P)\otimes_A B=k_B(Q)$, so
$$\Omega_{B/A} \otimes_B k_B(Q) \simeq  
\Omega_{B/A} \otimes_A k_A(P) \simeq \Omega_{k_A(P)/k_A(P)}=0
$$
cf. \cite[Lemme 1.13]{A}. 
By Nakayama's Lemma, we get that 
  $\Omega_{B/A}$ is null because it is  a finitely generated $B$-module.
By \cite[Proposition 1.5]{L}, we obtain that $u$ is an epimorphism.
\end{proof}

For further use, we give the following result.

\begin{proposition}\label{24}
Let $\mathcal{P}$ be one of the following three conditions: 
PSI-morphism, strong PSI-morphism, epimorphism.
An inductive limit of $\mathcal{P}$-morphisms  is a 
$\mathcal{P}$-morphism.
\end{proposition}
\begin{proof}
Let $u:A\rightarrow B$ be the limit of the inductive system of ring morphisms $\{u_i:A_i\rightarrow B_i\}_{i\in I}$. Let $Q$ be a prime ideal of $B$ and $P_i$, $Q_i$, $P$ the inverse image of $Q$ in $A_i$, $B_i$, $A$ respectively. If all $u_i$'s are PSI-morphisms, then 
$k_{A_i}(P_i)\otimes_{A_i} {B_i}= k_{B_i}(Q_i)$, so taking the limit we get
$k_{A}(P_i)\otimes_{A} {B}= k_{B}(Q)$. Hence $u$ is a PSI-morphism.
A similar argument works in the other two cases.
\end{proof}

\begin{corollary}\label{27}
An inductive limit of finite type strong PSI-morphisms is an epimorphism.
\end{corollary}
\begin{proof}
Combine Theorem \ref{301} and Proposition \ref{24}.
\end{proof}

In particular, the morphism of Example \ref{18} cannot be written as an inductive limit of finite type strong PSI-morphisms.

\begin{remark}
We remark the following extension of part $(ix)$ of Theorem \ref{3}.
Let $A\subseteq B$ be a ring extension and let $\Gamma$ denote the set of all  rings between  $A$ and $B$. Suppose that every $C\in \Gamma$ is integrally closed in $B$. By \cite[Theorem 5.2, page 47]{KZ}, 
$A\subseteq C$ is a flat epimorphism  (hence a strong PSI-extension) for each $C\in \Gamma$.  
\end{remark}

The next result is a strong PSI-morphism variant of Theorem \ref{8}.

\begin{theorem}\label{12}
Suppose that $u:A\rightarrow B$ and  $v:B\rightarrow C$ are ring morphisms.

$(i)$ If $u$ and $v$ are strong PSI-morphisms, then so is $vu$.

$(ii)$ If $vu$ is a strong PSI-morphism, then so is $v$.
\end{theorem}
\begin{proof}
Let $M$ be a prime ideal of $C$, $Q=v^{-1}(M)$, $P=u^{-1}(Q)$.
We have the residue field extensions 
$k_A(P)\rightarrow  k_B(Q)\rightarrow k_C(M).$
So $k_A(P) = k_C(M)$ iff $k_A(P) =  k_B(Q)$ and 
$k_B(Q) = k_C(M).$
 Apply Theorem \ref{13} and Theorem \ref{8}.
\end{proof}

\begin{corollary}\label{19}
If $u:A\rightarrow B$ and  $v:B\rightarrow C$ are strong PSI-morphisms, then so is $A\rightarrow B\otimes_A C$.
\end{corollary}
\begin{proof}
Combine Theorems \ref{16} and \ref{12}.
\end{proof}

\begin{corollary}\label{20}
Let $u:A\rightarrow B$, $v:B\rightarrow C$, $t:C\rightarrow D$, 
$w:B\rightarrow D$ be ring morphisms such that 
 $wu=tv$ and   $D=w(B)t(C)$.
If $u$ and  $v$ are strong PSI-morphisms, then so is $A\rightarrow D$.
\end{corollary}
\begin{proof}
By Corollary \ref{19}, $A\rightarrow B\otimes_A C$ is a strong PSI-morphism. As the canonical map $B\otimes_A C\rightarrow D$ is surjective, we are done.
\end{proof}

The next result is in the spirit of \cite[Proposition 3.4]{L}.

\begin{proposition}\label{26}
Let $u:A\rightarrow B$ be a ring morphisms. There exists a greatest subring $C$ of $B$ containing $u(A)$ such that $u:A\rightarrow C$ is a strong
PSI-morphism. Call $C$ the {\em strong PSI-closure }of $A$ in $B$.
\end{proposition}
\begin{proof}
$C$ is the directed union of the subrings $D$ of $B$ containing $u(A)$ such that $A\rightarrow D$ is a strong PSI-morphism, cf. Corollary \ref{20}. Apply Proposition \ref{24}.
\end{proof}

\begin{corollary}\label{25}
Let $A\subseteq B$ be an extension of domains where $A$ a Pr\"ufer domain with quotient field $K$. Then the PSI-closure of $A$ in $B$ is $K\cap B$.
\end{corollary}
\begin{proof}
Let $C$ denote the PSI-closure of $A$ in $B$. It is clear that $C\subseteq K\cap B$. On the other hand, since $A$ is a Pr\"ufer domain, $A\subseteq D$ is an epimorphism cf. \cite[Theorem 3.13, page 37]{KZ} and \cite[Corollary 6.5.19]{FHP}, hence a strong PSI-extension, for each overring $D$ of $A$. Thus $C = K\cap B$.
\end{proof}

\begin{example}
For a field $K$, the PSI-closure of $K[X]$ in $K[[X]]$ is contained in 
$K(X)\cap K[[X]]=K[X]_{(X)}$, cf. Corollary  \ref{25}.
\end{example}

We close our paper by giving two more constructions of PSI-morphisms
(Propositions  \ref{17} and \ref{170}).

\begin{proposition}\label{17}
Let $A\subseteq B$ be ring extension and $I$  a common ideal of $A$ and $B$. Then $A\subseteq  B$ is a PSI-extension (resp. strong PSI-extension) iff  $A/I\subseteq B/I$ is a PSI-extension (resp. strong PSI-extension).
\end{proposition}
\begin{proof}
The "only if part" follows from Theorem \ref{16}. For the "if part", let $P\in Spec(A)$. Suppose that $P\not\supseteq I$ and pick $f\in I-P$. Then $fB\subseteq A$, so $A_P=B_P$, hence $k_{A}(P)=k_{A}(P)\otimes_{A} B$.
To complete the proof, it suffices to see that when $P\supseteq I$, we have 
$k_{A/I}(P/I)=k_{A}(P)$ and 
$k_{A/I}(P/I)\otimes_{A/I} B/I = k_{A}(P)\otimes_{A} B$.
\end{proof}

\begin{corollary}
A ring extension $A\subseteq  B$  is a PSI-morphism (resp. strong PSI-morphism) iff $A+XB[X] \subseteq  B[X]$ is a PSI-morphism (resp. strong PSI-morphism).
\end{corollary}

We give a power series application of Proposition \ref{17} which extends part $(xii)$ of Theorem \ref{3}.

\begin{proposition}\label{29}
Let $A\subseteq B$ be ring extension such that some maximal ideal $M$ of $B$ is contained in $A$. Then  $A[[X]]\subseteq  B[[X]]$ is a PSI-extension.
\end{proposition}
\begin{proof}
As $M[[X]]$ is a common ideal of $A[[X]]$ and $B[[X]]$, it suffices to show that  $(A/M)[[X]]\subseteq  (B/M)[[X]]$ is a PSI-morphism,
cf. Proposition \ref{17}. In other words, 
it suffices to show that  $A[[X]]\subseteq  B[[X]]$ is a PSI-morphism
when $A$ is a domain and $B$ is a field. Since the prime ideals of $B[[X]]$ are $0$ and $XB[[X]]$, it suffices to see that
$$k_{A[[X]]}(0)\otimes_{A[[X]]} B[[X]] = B[[X]][1/X]=B((X))$$ and 
$$k_{A[[X]]}(XA[[X]])\otimes_{A[[X]]} B[[X]] = B[[X]]/(X)=B.$$ 
\end{proof}

\begin{proposition}\label{170}
Let $u:A\rightarrow B$, $v:A\rightarrow C$ be ring morphisms and $\alpha$, $\beta$ the images in $Spec(A)$ of $Spec(B)$, $Spec(C)$ respectively. Let 
$\mathcal{P}$ be one of the following three conditions: 
PSI-morphism, strong PSI-morphism, epimorphism. Then the diagonal map $w:A\rightarrow B\times C$ is $\mathcal{P}$-morphism iff 
 $u$, $v$ are $\mathcal{P}$ and $\alpha\cap \beta = \emptyset$.
\end{proposition}
\begin{proof}
The assertion is a consequence of the following three remarks. 
If $w$ is a (strong) PSI-morphism, compose it with the canonical projections  to get that $u$, $v$ are (strong) PSI-morphism, cf. Theorem \ref{8}.
For $P\in Spec(A)$, if
$$k_{A}(P)\otimes_{A} B\times C = k_{A}(P)\otimes_{A} B\times k_{A}(P)\otimes_{A} C$$  is a field, then $k_{A}(P)\otimes_{A} B=0$ or $k_{A}(P)\otimes_{A} C=0$. Note that $B\otimes_A C=0$, otherwise for any   prime ideal $P$ of $B\otimes_A C$,  the preimage of $P$ in $A$ belongs to $\alpha\cap \beta$. Thus $(B\times C)\otimes_A (B\times C)=(B\times C)$.
\end{proof}

For instance, the diagonal map 
$ \mathbb{Q}[[X]]\rightarrow \mathbb{R}((X))\times \mathbb{Q}$ is a PSI-morphism.


\end{document}